\documentclass[12pt,twoside,reqno]{amsart}

\usepackage{amsmath, amssymb, amsthm, enumitem}
\usepackage[hyperindex]{hyperref}
\usepackage[alphabetic]{amsrefs}

\newcommand{\IR}{\mathbb{R}}

\newcommand{\IN}{\mathbb{N}}

\newcommand{\eps}{\varepsilon}
\newcommand{\ov}[1]{\overline{#1}}
\newcommand{\td}[1]{\widetilde{#1}}

\DeclareMathOperator{\Ric}{Ric}
\DeclareMathOperator{\Rm}{Rm}

\DeclareMathOperator{\id}{id}
\DeclareMathOperator{\dist}{dist}

\DeclareMathOperator{\eucl}{eucl}

\newcommand{\EMPTY}[1]{}

\newtheorem{Theorem}{Theorem}
\newtheorem{Lemma}[Theorem]{Lemma}

\numberwithin{equation}{section}

\title[A Ricci flow proof of a result by Gromov]{A Ricci flow proof of a result by Gromov on lower bounds for scalar curvature}
\author{Richard H Bamler}
\address{Department of Mathematics, UC Berkeley, CA 94720, USA}
\email{rbamler@math.berkeley.edu}
\date{\today}

\begin{document}

\begin{abstract}
In this note we reprove a theorem of Gromov using Ricci flow.
The theorem states that a, possibly non-constant, lower bound on the scalar curvature is stable under $C^0$-convergence of the metric.
\end{abstract}

\maketitle

\section{Introduction}
In \cite{Gromov}, Gromov proved the following theorem as an application of his theory of singular spaces with positive scalar curvature:

\begin{Theorem} \label{Thm:mainthm}
Let $M$ be a (possibly open) smooth manifold and $\kappa : M \to \IR$ a continuous function.
Consider a sequence of $C^2$ Riemannian metrics $g_i$ on $M$ that converges to a $C^2$ Riemannian metric $g$ in the local $C^0$-sense.
Assume that for all $i = 1, 2, \ldots$ the scalar curvature of $g_i$ satisfies $R (g_i) \geq \kappa$ everywhere on $M$.
Then $R (g) \geq \kappa$ everywhere on $M$.
\end{Theorem}

In this note, we present an independent proof of this theorem using Ricci flow.

Let us first sketch the idea of our proof in the most elementary setting:
Consider the case in which $M$ is closed and $\kappa$ is constant on $M$ and assume that all our metrics, $g, g_1, g_2, \ldots$, are smooth.
We may then solve the Ricci flow equation
\begin{equation} \label{eq:RF}
 \partial_t g_t = - 2 \Ric (g_t), 
\end{equation}
starting with the limiting metric $g_0 = g$ for some short time and obtain a smooth solution $(g_t)_{t \in [0,\tau^*)}$ for some $\tau^* > 0$.
By a theorem of Koch and Lamm (cf \cite[sec 5.3]{Koch-Lamm-II}), there are constants $\eps = \eps ( g ) > 0$ and $\tau^* >  \tau = \tau (g) > 0$ such that any smooth metric $g'$ on $M$ that is $(1+\eps)$-bilipschitz to $g$ can be evolved into a smooth Ricci flow $(g')_{t \in [0, \tau)}$ on the uniform time-interval $[0, \tau)$.
So the metrics $g_i$, for sufficiently large $i$, can be evolved into a Ricci flow $(g_{i,t})_{t \in [0, \tau)}$ on $[0, \tau)$.
By the weak maximum principle applied to the evolution equation for the scalar curvature,
\begin{equation} \label{eq:scalarcurvevol}
 \partial_t R (g_{i,t}) = \Delta R (g_{i,t}) + 2 |{\Ric} (g_{i,t}) |^2,
\end{equation}
it follows that $R (g_{i,t}) \geq \kappa$ for all $t \in [0, \tau)$, for sufficiently large $i$.
Using again the results of Koch and Lamm, one can see that the Ricci flows $(g_{i,t})_{t \in [0, \tau)}$ converge locally smoothly to $(g_t)_{t \in [0, \tau)}$ on $M \times (0,\tau)$, modulo diffeomorphisms.
So we have $R (g_t) \geq \kappa$ for all $t \in (0, \tau)$.
Letting $t$ go to zero yields $R(g) = R(g_0) \geq \kappa$.

The difficulties in the following proof arise from the general form of the theorem:
$M$ need not be closed, $\kappa$ need not be constant and the metrics $g, g_1, g_2, \ldots $ have only regularity $C^2$.
We will overcome the first two difficulties by localizing the argument presented in the previous paragraph.
This localization will be the main challenge of this note.
The key estimate for this localization can be found in Lemma \ref{Lem:scalestimate}.
The issue concerning the $C^2$-regularity of the metrics can be addressed by considering the Ricci DeTurck flow instead of the Ricci flow.

I would like to thank John Lott for pointing out Gromov's theorem to me.

\section{Ricci DeTurck flow} \label{sec:Preliminaries}
The Ricci DeTurck flow, as introduced in \cite{DeTurck}, differs from the Ricci flow by a family of diffeomorphisms.
Its evolution equation is strongly parabolic, as opposed to the Ricci flow equation, which is only weakly parabolic.
This parabolicity is a consequence of a gauge fixing, which we will recall in the following.
Choose a background metric $\ov{g}$ on $M$ and define the Bianchi operator
\[ X_{\ov{g}}^i ( h ) = (\ov{g} + h)^{ij} (\ov{g} + h)^{pq} \big( - \nabla^{\ov{g}}_p h_{qj} + \tfrac12 \nabla^{\ov{g}}_j h_{pq} \big) , \]
which assigns a vector field to every symmetric $2$-form $h$ on $M$.
The evolution equation of the Ricci DeTurck flow now reads
\begin{equation} \label{eq:RdTflow}
 \partial_t g_t = - 2 \Ric (g_t) - \mathcal{L}_{X (g_t)} g_t.
\end{equation}
The evolution equation of the difference $h_t := g_t - \ov{g}$ takes the form
\[ \partial_t h_t = \triangle h_t + 2 \Rm * h_t + Q ( h_t, \nabla h_t, \nabla^2 h_t), \]
where
\begin{multline} \label{eq:Q}
Q ( h_t, \nabla h_t, \nabla^2 h_t) =  (\ov{g} + h_t) ^{-1} * (\ov{g} + h_t) ^{-1} * \nabla h_t * \nabla h_t  \\
+ \big( (\ov{g} + h_t) ^{-1} - \ov{g}^{-1} \big) * \nabla^2 h_t . 
\end{multline}
(Here all covariant derivatives are taken with respect to $\ov{g}$.)
So if $h_t$ is $C^0$-close to $\ov{g}$, then (\ref{eq:RdTflow}) is strongly parabolic.

Given a solution $(g_t)_{t \in I}$ of the Ricci DeTurck flow equation (\ref{eq:RdTflow}), we can construct a solution $(\td{g}_t)_{t \in I}$ to the Ricci flow equation (\ref{eq:RF}) by pulling back via a family of diffeomorphisms as follows:
Let $(\Phi_t)_{t \in I}$ be a flow of the time-dependent family of vector fields $X_{\ov{g}} (g_t)$, meaning that
\begin{equation} \label{eq:Phi}
 \partial_t \Phi_t = X_{\ov{g}} (g_t) \circ \Phi_t.
\end{equation}
Then $\td{g}_t := \Phi^*_t g_t$ satisfies the Ricci flow equation (\ref{eq:RF}).

We will also need to use the heat kernel on a Ricci flow and Ricci DeTurck flow background.
Consider first the heat kernel $\td{K} (x,t; y,s)$, $s < t$, on a Ricci flow background $(\td{g}_t = \Phi^*_t g_t)_{t \in I}$, that is for fixed $(y,s) \in M \times I$ we have
\[ \partial_t \td{K} (x, t; y, s) = \Delta_{\td{g}_t, x} \td{K} (x,t; y,s) \]
and $\td{K} (\cdot, t; y,s)$ approaches a $\delta$-function centered at $y$ as $t \searrow s$.
Then, for fixed $(x,t) \in M \times I$, the function $\td{K} (x,t; \cdot, \cdot )$ is a kernel of the conjugate heat equation
\[ - \partial_s \td{K} (x,t ;  y, s) = \Delta_{\td{g}_s, y} \td{K} (x,t; y,s) - \td{R}(y,s) \td{K} (x,t; y,s). \]
Here $\td{R}(y,s)$ denotes the scalar curvature of $\td{g}_s$ in $y$.
Note that this equation implies that for all $s < t$
\begin{equation} \label{eq:KintegralRF}
 \int_M \td{K} (x,t; y,s) d\td{g}_s (y) = 1 
\end{equation}

Consider now the push-forward $K(x,t; y,s)$ of $\td{K}(x,t; y,s)$ under $\Phi_t$.
That is
\[ K(x,t; y,s) := \td{K} (\Phi^{-1}_t (x), t; \Phi^{-1}_s (y), s). \]
This kernel is the associated heat kernel on the Ricci DeTurck flow background $(g_t)_{t \in I}$ and it satisfies
\begin{equation} \label{eq:heateqRdT}
 \partial_t K (x,t; y,s) = \Delta_{g_t, x} K (x,t; y,s) - \partial_{X_{\ov{g}} (g_t), x}  K(x, t; y,s) 
\end{equation}
for fixed $(y,s) \in M \times I$, as well as
\[ - \partial_s K (x,t ;  y, s) = \Delta_{\td{g}_s, y} K (x,t; y,s) - R(y,s) K (x,t; y,s) + \partial_{X_{\ov{g}} (g_s), y} K(x, t; y,s). \]
for fixed $(x,t) \in M \times I$.
As a direct consequence of (\ref{eq:KintegralRF}), we also obtain
\begin{equation} \label{eq:KintegralRFdT}
\int_M K (x,t; y,s) dg_s (y) = 1.
\end{equation}

\section{Proof}
In the following, we will fix some dimension $n \geq 2$ of the manifold $M$ and we will not mention this dependence anymore.
We will also frequently consider Euclidean space $\IR^n$ with the standard Euclidean metric $g_{\eucl}$ and origin $o \in \IR^n$.

Let us first establish and recall a short-time existence result for Ricci DeTurck flows, which is mainly a consequence of the work of Koch and Lamm (cf \cite{Koch-Lamm}) and which will become important for us.
It states that metrics that are sufficiently close to the Euclidean metric in the $C^0$-sense can be evolved by the Ricci DeTurck flow on a uniform time-interval.
This flow becomes instantly smooth and depends continuously on the initial data.
Note that we have phrased the following lemma specifically such that it can be applied to our situation.
In fact, with some additional work, the lemma can be strengthened in several aspects:
For example, the Ricci DeTurck flow $(g_t)_{t \in [0,1)}$ can actually be extended to the time-interval $[0, \infty)$, the condition that $g - g_{\eucl}$ is compactly supported is not necessary and we have bounds on all higher derivatives of $g_t$.

\begin{Lemma} \label{Lem:KochLamm}
There are constants $\eps > 0$, $C_1 < \infty$ such that the following is true:

Consider a Riemannian metric $g$ on $\IR^n$ of regularity $C^2$ that is $(1+\eps)$-bilipschitz close to the standard Euclidean metric $g_{\eucl}$ and assume that $g - g_{\eucl}$ is compactly supported.
Then there is a continuous family of Riemannian metrics $(g_t)_{t \in [0, 1)}$ on $\IR^n$ such that the following holds:
\begin{enumerate}[label=(\alph*)]
\item For all $t \geq 0$, the metric $g_t$ is $1.1$-bilipschitz to $g_{\eucl}$.
\item $(g_t)$ is smooth on $\IR^n \times (0,1)$ and the map $[0,1) \to C^2 (\IR^n)$, $t \mapsto g_t$ is continuous.
\item $g_0= g$ and $(g_t)_{t \in (0, 1)}$ is a solution to the Ricci DeTurck equation
\begin{equation} \label{eq:RdTflow-eucl}
 \partial_t g_t = - 2 \Ric (g_t) - \mathcal{L}_{X_{g_{\eucl}} (g_t)} g_t.
\end{equation}
\item For any $t > 0$ and any $m = 0, \ldots, 10$ we have
\[ |\partial^m g_t| < \frac{C_1}{t^{m/2}}. \]
\item If $(g_{i,t})_{t \in [0,1)}$ is a sequence of solutions to (\ref{eq:RdTflow-eucl}) that are continuous on $\IR^n \times [0,1)$ and smooth on $\IR^n \times (0,1)$ and if $g_{i,0}$ converges to some metric $g_{0}$ uniformly in the $C^0$-sense, then $(g_{i,t})$ converges to $(g_t)$ uniformly in the $C^0$-sense on $\IR^n \times [0,1)$ and locally in the smooth sense on $\IR^n \times (0,1)$.
\end{enumerate}
\end{Lemma}

\begin{proof}
The lemma essentially follows from the work of Koch and Lamm (cf \cite[sec 4]{Koch-Lamm}).
In their paper, the authors analyze and solve (\ref{eq:RdTflow-eucl}) by rewriting the term $Q$ in (\ref{eq:Q}) as
\[ Q( h_t, \nabla h_t, \nabla^2 h_t) = R_1 ( h_t, \nabla h_t ) + \nabla^* R_2 ( h_t, \nabla h_t ), \]
where
\[ R_1 (h_t, \nabla h_t ) = (g_{\eucl} + h_t) ^{-1} * (g_{\eucl} + h_t) ^{-1} * \nabla h_t * \nabla h_t  \]
and $R_2 (h_t, \nabla h_t)$ is a $(0,3)$-tensor, which has the form
\[ R_2 (h_t, \nabla h_t ) = \big( (g_{\eucl} + h_t) ^{-1} - g_{\eucl}^{-1} \big) * \nabla h_t. \]
(Here and in the rest of the proof, all covariant derivatives are taken with respect to $g_{\eucl}$.)
So the evolution equation for $h_t = g_t - g_{\eucl}$ becomes
\begin{equation} \label{eq:generalform}
 \partial_t h_t = \Delta h_t + R_1 (h_t, \nabla h_t ) + \nabla^* R_2 (h_t, \nabla h_t).
\end{equation}

The existence of $(g_t)_{t \in [0, 1)}$ and assertions (a), (c) and (d) are consequences of \cite[Theorem 4.3]{Koch-Lamm}, which follows from a general analysis of equations of the form (\ref{eq:generalform}).
This theorem also provides the bound
\begin{equation} \label{eq:C0bound}
 \Vert h_t \Vert_{C^0(\IR^n \times [0,1))} \leq C \Vert h_0 \Vert_{C^0(\IR^n)}
\end{equation}
as well as bounds on the derivatives of $h_t$ on $\IR^n \times (0,1)$.

Likewise, one may look at two different solutions, $(g^1_t)$ and $(g^2_t)$, of (\ref{eq:RdTflow-eucl}) and find that for any $a > 0$, the multi-valued function $(g^1_t - g_{\eucl}, a (g^1_t - g^2_t))$ satisfies an equation of a form similar to that of (\ref{eq:generalform}).
So if $a > 0$ is chosen small enough such that
\[ \Vert g^1_0 - g_{\eucl} \Vert_{C^0(\IR^n)} + a \Vert g^1_0 - g^2_0 \Vert_{C^0(\IR^n)} < \eps', \]
for some universal $\eps' > 0$, then one can derive the bound, which is similar to (\ref{eq:C0bound}):
\begin{multline*}
 \Vert g^1_t - g_{\eucl} \Vert_{C^0(\IR^n \times [0,1))} + a \Vert g^1_t - g^2_t \Vert_{C^0(\IR^n \times [0,1))} \\
  \leq C' \big( \Vert g^1_0 - g_{\eucl} \Vert_{C^0(\IR^n)} + a \Vert g^1_0 - g^2_0 \Vert_{C^0(\IR^n)} \big).
\end{multline*}
So if $\Vert g^1_0 - g_{\eucl} \Vert_{C^0(\IR^n)} < \frac{\eps}2$, then we may choose $a := \frac{\eps}2 \Vert g^1_0 - g^2_0 \Vert_{C^0(\IR^n)}^{-1}$ and deduce
\begin{equation} \label{eq:lipschitz}
 \Vert g^1_t - g^2_t \Vert_{C^0(\IR^n \times [0,1))} \leq 2 C' \Vert g^1_0 - g^2_0 \Vert_{C^0(\IR^n)}.
\end{equation}
This implies assertion (d).

For assertion (b), observe that (\ref{eq:lipschitz}) implies that difference quotients of $(g_t)$ are uniformly bounded.
So we obtain a uniform bound on $\nabla g_t$.
Similarly, we obtain a uniform bound on $\nabla^2 g_t$.
So
\[  \partial_t \nabla^2 h_t  = \Delta \nabla^2 h_t  + R_{1,t} + \nabla^* R_{2,t}, \]
where
\[ |R_{1,t}| < C'' |\nabla^3 h_t| + C'' \qquad \text{and} \qquad |R_{2,t}| < C'' |h_t| \cdot |\nabla^3 h_t| + C''. \]
The continuity of $\nabla^2 h_t$, then follows similarly as the continuity for $h_t$ in \cite[Theorem 4.3]{Koch-Lamm}.
\end{proof}

Next, we analyze the heat kernel $K(x,t; y,s)$ on a Ricci DeTurck flow background, as introduced in section \ref{sec:Preliminaries}.
Our main observation will be that, on a small time-interval, the kernel can be bounded from above by a standard Gaussian.

\begin{Lemma} \label{Lem:hkestimate}
For any $A < \infty$ there are constants $C_2 = C_2 (A), D = D(A) < \infty$ and $0 < \theta = \theta (A) < \frac12$ such that the following is true:

Let $(g_t)_{t \in [0,\theta]}$ be a smooth solution to the Ricci DeTurck equation (\ref{eq:RdTflow-eucl}) on $\IR^n$.
Assume that $g_t$ is $1.1$-bilipschitz to $g_{\eucl}$ for all $t \in [0, \theta]$ and assume that $| \partial^m g_t | < A$ for all $m = 0, \ldots, 10$.
Consider the heat kernel $K(x,t; y,s)$ on $\IR^n \times [0,\theta]$ as discussed in section \ref{sec:Preliminaries}.
Then, for any $(x,t), (y,s) \in \IR^n \times [0, \theta]$ with $s < t$, we have
\[ K(x,t; y,s ) < \frac{C_2}{(t-s)^{n/2}} \exp \Big( { - \frac{ d_{g_{\eucl}}^2 (x, y)}{D (t-s)}} \Big) \]
and for any $r > 0$
\[ \int_{\IR^n \setminus B (x, r)} K(x,t; y,s) dg_s (y)  < C_2 \exp \Big( { - \frac{ r^2}{D (t-s)}} \Big). \]
Here $B(x,r)$ denotes the $r$-ball in $\IR^n$ with respect to the Euclidean metric $g_{\eucl}$.
\end{Lemma}

We remark that we can actually choose $D > 4$ arbitrarily as in the work of Cheng, Li and Yau (cf \cite{CLY}).

\begin{proof}
Consider the associated Ricci flow $\td{g}_t = \Phi^*_t g_t$, where $\Phi_t$ is defined as in (\ref{eq:Phi}) with $\Phi_0 = \id_{\IR^n}$ and the heat kernel $\td{K} (x,t; y,s) = K(\Phi_t^{-1} (x), t; \Phi^{-1}_s (y), s)$ on a Ricci flow background.
Using the derivative bounds on $g_t$, we find that for small enough $\theta$, the metrics $\td{g}_t$ stay $1.2$-bilipschitz close to $g_{\eucl}$.
Moreover, we can find a constant $C' = C'(A) < \infty$ such that for all $x \in \IR^n$ and $s, t \in [0, \theta]$
\[ \dist_{g_{\eucl}} (\Phi_t(x),\Phi_s(x)) < C' |s-t|. \]

Using \cite[Theorem 26.25]{CetalIII} and the derivative bounds on $g_t$, we get that for sufficiently small $\theta$
\[ \td{K} (x,t; y,s ) < \frac{C}{(t-s)^{n/2}} \exp \Big( { - \frac{ d_{g_{\eucl}}^2 (x, y)}{D (t-s)}} \Big). \]
Here $C = C(A), D= D(A) <\infty$ are some uniform constants.
So
\begin{alignat*}{1}
 K(x,t; y,s ) &< \frac{C}{(t-s)^{n/2}} \exp \Big( { - \frac{ d_{g_{\eucl}}^2 (\Phi_t(x), \Phi_s(y))}{D (t-s)}} \Big) \\
& \leq \frac{C}{(t-s)^{n/2}} \exp \Big( { - \frac{ d_{g_{\eucl}}^2 (\Phi_t(x), \Phi_t(y)) - d_{g_{\eucl}}^2 (\Phi_t(y), \Phi_s(y)) }{2D (t-s)}} \Big) \\
& \leq \frac{C}{(t-s)^{n/2}} \exp \Big( { - \frac{ d_{\td{g}_t}^2 (x,y) - C' (t-s)^2 }{2D (t-s)}} \Big) \\
& \leq \frac{C''}{(t-s)^{n/2}} \exp \Big( { - \frac{ d_{g_{\eucl}}^2 (x,y) }{2(1.2)^2 D (t-s)}} \Big)
\end{alignat*}
This proves the first assertion of the lemma, after adjusting $D$.
The second assertion follows by integration and adjusting $D$ again.
\end{proof}

We now state and prove the our key estimate:

\begin{Lemma} \label{Lem:scalestimate}
There is an $\eps > 0$ and for every $\delta > 0$ there is a $\tau = \tau ( \delta) > 0$ such that the following is true:

Let $g_0$ be a $C^2$-Riemannian metric on $\IR^n$ that is $(1+ \eps)$-bilipschitz close to the standard Euclidean metric $g_{\eucl}$ and assume that $g_0 - g_{\eucl}$ is compactly supported.
Assume that $R (g_0) > a$ on $B (o, 1)$ for some $a \in \IR$.
Then there is a solution $(g_t)_{t \in [0,1)}$ to the Ricci DeTurck flow equation (\ref{eq:RdTflow-eucl}) with initial metric $g_0$ such that $R (o, t) > a - \delta$ for all $t \in [0, \tau]$.
\end{Lemma}

\begin{proof}
Choose $\eps > 0$ from Lemma \ref{Lem:KochLamm}.
Then $g_0$ can be evolved to a solution $(g_t)_{t \in [0,1)}$ of the Ricci DeTurck flow (\ref{eq:RdTflow-eucl}) such that $g_t$ is $1.1$-bilipschitz to $g_{\eucl}$ for all $t \in [0,1)$.
By Lemma \ref{Lem:KochLamm}(d), we can find a uniform constant $C_3 < \infty$ such that for all $t \in (0,1)$
\begin{equation} \label{eq:C3}
 |{\Rm}| (\cdot, t), \; |R|(\cdot, t) < \frac{C_3}t.
\end{equation}
Note that the scalar curvature bound follows already from the bound on the Riemannian curvature and is only mentioned for convenience.
We even have the more precise lower bound $R (\cdot, t) > - \frac{n}{2t}$, which will, however, not be essential for us.

The scalar curvature $R$ of $g_t$ satisfies the equation
\[ \partial_t R = \Delta R - \partial_{X_{g_{\eucl}} (g_t)}  R + 2 |{\Ric}|^2, \]
which is the analogue of (\ref{eq:scalarcurvevol}) for Ricci DeTurck flow.
So the scalar curvature is a supersolution for the associated heat equation on a Ricci DeTurck flow background (compare with (\ref{eq:heateqRdT})).
Hence it follows that for any $x \in \IR^n$ and any $0 < s < t \leq 1$
\begin{equation} \label{eq:Rxtlowerbound}
 R(x,t) \geq \int_{\IR^n} K(x,t; y,s) R(y,s) dg_s(y). 
\end{equation}
Alternatively, this equation can be derived from the corresponding equations involving the heat kernel $\td{K} (x,t; y,s)$ on a Ricci flow background and then pulling back via the diffeomorphisms $\Phi_t$.

Let us now choose the constants that we will be using throughout the proof.
Let $C_1 < \infty$ be the constant from Lemma \ref{Lem:KochLamm} and let $\theta := \theta (2C_1) > 0$ and $D := D (2 C_1), C_2 := C_2 (2C_1) < \infty$ be the constants from Lemma \ref{Lem:hkestimate}.
Next choose $\beta > 0$ small enough such that
\[ 1-\beta > \sqrt{1-\theta} \]
and set
\[ c:= \frac{\beta^2}{D\theta}. \]
Choose $\lambda > 0$ such that
\[ 1+ \lambda < \frac{(1-\beta)^2}{1-\theta} . \]
Note that
\[ \sum_{k = 1}^\infty \frac{2C_2 C_3}{(1-\theta)^k} \exp \big( { - c (1+ \lambda)^k } \big) < \infty. \]
So for any $\delta > 0$, we can find a large number $N = N(\delta) \in \IN$ such that
\begin{equation} \label{eq:lessthandelta}
 \sum_{k = N}^\infty \frac{2C_2 C_3}{(1-\theta)^k} \exp \big( { - c (1+ \lambda)^k } \big) < \delta. 
\end{equation}
We can finally define
\[ \tau = \tau (\delta ) := (1- \theta)^N. \]

Next, we choose times and radii
\[ t_k := (1 - \theta)^k, \qquad \text{and} \qquad r_k := 1 - (1- \beta)^{k}, \]
for $k = 1, 2, \ldots$ and set
\[ a_k := \inf_{B(o, r_k)} R(\cdot, t_k). \]
By (\ref{eq:C3}), we have
\begin{equation} \label{eq:akC3}
 |a_k| \leq \frac{C_3}{t_{k}} 
\end{equation}
and by Lemma \ref{Lem:KochLamm}(b), we find that
\begin{equation} \label{eq:liminf}
 \liminf_{k \to \infty} a_k \geq a. 
\end{equation}

We will now estimate $a_k$ from below in terms of $a_{k+1}$.
Let $x \in B (o, r_k)$.
First note that by Lemma \ref{Lem:hkestimate}, we have
\begin{multline*}
 \int_{\IR^n \setminus B( x, r_{k+1} - r_k ) } K(x,t_k ; y, t_{k+1} ) dg_{t_{k+1}} (y)
 < C_2 \exp \Big({ - \frac{(r_{k+1} - r_k)^2}{D  (t_k- t_{k+1})} }\Big) \\
 < C_2 \exp \Big({ - \frac{  \beta^2 (1-\beta)^{2k}  }{D \theta (1- \theta)^k} }\Big)  
 < C_2 \exp \big({ - c (1+\lambda)^k } \big).
\end{multline*}
Then, by (\ref{eq:Rxtlowerbound}), (\ref{eq:C3}), (\ref{eq:KintegralRFdT}) and (\ref{eq:akC3})
\begin{alignat*}{1}
R(x, t_k) &\geq \int_{\IR^n} K(x,t_k; y, t_{k+1} ) R(y, t_{k+1}) dg_{t_{k+1}} (y) \\
&\geq a_{k+1} \int_{B (x, r_{k+1} - r_k)} K(x,t_k; y, t_{k+1} )  dg_{t_{k+1}} (y) \\
&\qquad\qquad
- \frac{C_3}{t_{k+1}} \int_{\IR^n \setminus B(x, r_{k+1} - r_k)} K(x,t_k; y, t_{k+1} ) dg_{t_{k+1}} (y) \\
&\geq a_{k + 1} - \Big( \frac{C_3}{t_{k+1}} + a_{k+1} \Big) \int_{\IR^n \setminus B (x, r_{k+1} - r_k)} K(x,t_k; y, t_{k+1} ) dg_{t_{k+1}} (y) \\
&\geq a_{k+1} - \frac{2C_3}{t_{k+1}} \cdot C_2 \exp \big( { - c (1+ \lambda)^k } \big).
\end{alignat*}
So we conclude that
\[ a_k \geq a_{k+1} - \frac{2C_2 C_3}{(1-\theta)^k} \exp \big( { - c (1+ \lambda)^k } \big). \]
Together with (\ref{eq:liminf}) and (\ref{eq:lessthandelta}), this implies that for all $k \geq N$
\[ R(o, t_k) \geq a_k > a - \delta. \]
In particular, this proves the claim for $t = \tau$ and for $t = t_k$ for all $k \geq N$.
The lower bound on $R(o, t)$ for any $t \in [0, \tau]$ follows similarly, e.g. by perturbing the parameter $\theta$ slightly or by parabolic rescaling with a bounded factor.
For our purposes, it is, however, enough to know the bound $R(o, t_k) > a - \delta$ for all $k \geq N$.
\end{proof}

We can finally prove Theorem \ref{Thm:mainthm}.

\begin{proof}[Proof of Theorem \ref{Thm:mainthm}]
Consider the constant $\eps > 0$ from Lemma \ref{Lem:scalestimate}.
Assume that for some $x \in M$ and some $\kappa' \in \IR$ we have $R (g, x) < \kappa' < \kappa (x)$.
By restricting $M$ to a subset, we may assume that for some $\kappa'' \in \IR$
\[ R (g, x) < \kappa'  < \kappa'' <  \kappa (y) \qquad \text{for all} \qquad x, y \in M. \]
Next, we choose a chart $\Phi : U \subset M \to \IR^n$ around $x$ such that $\Phi (x) = o$ and such that $\Phi_* g$ on $\Phi (U)$ is $(1+\eps)$-bilipschitz to $g_{\eucl}$.
By rescaling, we may assume that $\Phi (U)$ contains the closure of $B(o, 1)$.

Let $\varphi : \IR^n \to [0,1]$ be a smooth cutoff function that is constantly equal to $1$ on $B(o, 1)$ and whose support is contained in $\Phi(U)$.
Consider the metric
\[ g_0 := \varphi \Phi_* g + (1- \varphi ) g_{\eucl}. \]
This metric is still $(1+\eps)$-bilipschitz to $g_{\eucl}$ and $C^2$ and satisfies
\begin{equation} \label{eq:lessthankappas}
 R (g_0,  o) < \kappa' .
\end{equation}
Similarly, the metrics
\[ g_{i, 0} := \varphi \Phi_* g_i + (1- \varphi ) g_{\eucl} \]
are $(1+\eps)$-bilipschitz close to $g_{\eucl}$ and $C^2$ and satisfy
\[ R (g_{i, 0}) > \kappa'' \qquad \text{on} \qquad B(o, 1).\]
Moreover, $g_{i,0}$ converge to $g_0$ uniformly in the $C^0$-sense.

Apply Lemma \ref{Lem:scalestimate} to each metric $g_{i, 0}$ with $a = \kappa''$ and $\delta = \frac12 (\kappa'' - \kappa')$.
Then we get a sequence of Ricci DeTurck flows $(g_{i, t})_{t \in [0, 1)}$ such that for any $t \in [0, \tau (\delta)]$ we have
\[ R ( g_{i,t},  o) > \kappa'' - \delta . \]
By Lemma \ref{Lem:KochLamm}(e), these Ricci DeTurck flows converge to the Ricci DeTurck flow $(g_t)_{t \in [0,1)}$ starting from $g_0$.
The convergence is uniformly $C^0$ on $\IR^n \times [0,1)$ and locally smooth on $\IR^n \times (0,1)$.
So for all $t \in (0, \tau]$
\[ R (g_t, o) \geq \kappa'' - \delta. \]
Using Lemma \ref{Lem:KochLamm}(b), it follows that
\[ R (g_0, o) = \lim_{t \to 0} R (g_{0,t}, o) \geq \kappa'' - \delta > \kappa', \]
in contradiction to (\ref{eq:lessthankappas}).
\end{proof}

\end{document}